\documentclass[12pt]{amsart}
\usepackage{fullpage,url,amssymb,enumerate,colonequals}

\usepackage[utf8]{inputenc}

\usepackage{mathrsfs} 
\usepackage{MnSymbol}
\usepackage{extarrows}
\usepackage{lscape}
\usepackage[all,cmtip]{xy}

\usepackage{amscd}
\usepackage[justification=centering]{caption}

\usepackage[OT2,T1]{fontenc}

\usepackage[dvipsnames]{xcolor}

\usepackage[
    colorlinks,
    linkcolor=MidnightBlue,  
    citecolor=MidnightBlue,  
    urlcolor=MidnightBlue,   
    backref,
    pdfauthor={Imin Chen}, 
]{hyperref}

\usepackage{comment}
\usepackage{float}

\usepackage[most]{tcolorbox}

\newcommand{\F}{\mathbb{F}}

\newcommand{\PP}{\mathbb{P}}
\newcommand{\Q}{\mathbb{Q}}

\newcommand{\Z}{\mathbb{Z}}

\newcommand{\rhobar}{{\overline{\rho}}}

\newcommand{\calO}{\mathcal{O}}

\newcommand{\Fp}{\mathfrak{p}}
\newcommand{\Fq}{\mathfrak{q}}

\def\id#1{{\mathfrak{#1}}}      

\DeclareMathOperator{\End}{End}

\DeclareMathOperator{\Spec}{Spec}

\newcommand{\GL}{\operatorname{GL}}

\newcommand{\PSL}{\operatorname{PSL}}
\newcommand{\SL}{\operatorname{SL}}

\numberwithin{equation}{section}

\newtheorem{lemma}[equation]{Lemma}
\newtheorem{corollary}[equation]{Corollary}
\newtheorem{proposition}[equation]{Proposition}

\theoremstyle{definition}
\newtheorem{definition}[equation]{Definition}

\theoremstyle{remark}
\newtheorem{remark}[equation]{Remark}

\definecolor{darkgreen}{rgb}{0,0.5,0}

\title{A note on conductors of Frey representations at $2$}

\date{\today}

\author{Imin Chen}
\address{Department of Mathematics, Simon Fraser University\\
Burnaby, BC V5A 1S6, Canada } \email{ichen@sfu.ca}

\author{Lucas Villagra Torcomian}
\address{Department of Mathematics, Simon Fraser University\\
Burnaby, BC V5A 1S6, Canada } \email{lvillagr@sfu.ca}


\begin{document}

\maketitle

\begin{abstract}
In 2000, Darmon introduced the notion of Frey representations within the framework of the modular method for studying the generalized Fermat equation. A central step in this program is the computation of their conductors, with the case at the prime $2$ presenting particular challenges. In this article we study the conductor exponent at $2$ for Frey representations of signatures $(p,p,r)$, $(r,r,p)$, $(2,r,p)$, and $(3,5,p)$, all of which have hyperelliptic realizations. In particular we are able to determine the conductor at $2$ for even degree Frey representations of signature $(p,p,r)$ and $(3,5,p)$ and all rational parameters.
\end{abstract}


\section{Introduction}

In \cite{Darmonprogram}, Darmon introduced the notion of a Frey representation
\begin{equation*}
    \rho(t) : G_{K(t)} \rightarrow \GL_2(\bar{\F}_p),
\end{equation*}
which can be thought of as a family of two-dimensional residual Galois representations of $G_K$ varying with a parameter $t$, where $K$ is a totally real field, and $G_{K(t)}$ and $G_K$ denote the absolute Galois group of $K(t)$ and $K$, respectively.

Let $\id{q}_2$ be a prime in $K$ above 2. In this paper, we study the conductor exponent at $\id{q}_2$ of Frey representations
constructed from  different families of hyperelliptic curves parametrized by $t\in\Q$.

The methods rely on explicit computations of hyperelliptic models and the particular algebraic symmetries present in each hyperelliptic realization. 
In each case we cover, we obtain explicit hyperelliptic models over a suitable extension of $\Q_2$. As the hyperelliptic models are readily verified to be semistable, the examples in this paper may inform more sophisticated approaches to semistable reduction of hyperelliptic curves over $\Q_2$.

Frey representations have a natural application to the study of generalized Fermat equations, which are of the form 
\begin{align}\label{eq:GFE}
Ax^{n_1}+By^{n_2}=Cz^{n_3}, 
\end{align}
where $A, B, C$ are fixed non-zero integers and $n_i$ are positive integers. The triple of exponents $(n_1,n_2,n_3)$ is called the \textit{signature} of (\ref{eq:GFE}) and if we fix it such that $ \sum_{i=1}^3 1/n_i<1$, it is known that (\ref{eq:GFE}) has finitely many \textit{non-trivial primitive} solutions, i.e.\ solutions $(a,b,c)\in\Z^3$ where $abc\neq0$ and $\gcd(a,b,c)=1$ \cite[Theorem 2]{DarmonGranville}. 

In \cite{Darmonprogram}, Darmon introduced a novel approach to study (\ref{eq:GFE}). When following Darmon's program one specializes Frey representations to specific values of $t \in \Q$. More concretely, if $(a,b,c)$ is a non-trivial primitive solution to (\ref{eq:GFE}), one specializes at $t=Aa^{n_1}/Cc^{n_3}$.

By adopting this framework, one needs to know the conductor of the Frey representations. The particular case where the representation can be constructed from a hyperelliptic curve is of special interest, since there are different tools to compute the conductor of this. In a series of recent papers (see \cite{CelineClusters,azon2025,cazorla2025}), the conductor exponent at odd primes has been computed for all signatures which are known so far to admit hyperelliptic realizations, namely $(n_1, n_2, n_3)\in\{(p,p,r),(r,r,p), (2,r,p) \}$.

The aim of the present article is to complement the current state of the art by providing the computation of the conductor exponent at $\id{q}_2$ for these cases.  In addition, a  central contribution of this paper is the introduction of a new Frey hyperelliptic curve with signature $(3,5,p)$, together with a complete determination of its conductor exponent at $\id{q}_2$. So far, this signature has remained unexplored, and our work  represents a fundamental step towards its study through Darmon's program.

It is worth mentioning that for signatures $(p,p,r)$ and $(r,r,p)$, the conductor exponent at $\id{q}_2$ was determined in \cite{azon2025} for those values of $t$ coming from a solution to (\ref{eq:GFE}). As a complement, in the present article we focus on an arbitrary parameter $t\in\Q$. 

Subject to natural $2$-adic conditions on $t$, in Table \ref{table} we summarize our main results. For the prime $2$, three cases arise: $v_2(t)>0$, $v_2(1-t)>0$, or $v_2(t)<0$. Moreover, $v_2(t)<0$ if and only if $v_2(1-t)<0$, and in this situation we have $v_2(1-t)=v_2(t)$. Therefore, we will naturally split the analysis into these three cases.

\begin{table}[H]
\begin{tabular}{|c||c|c|c|c|}
\hline
Signature & Degree & $v_2(t)$ & $v_2(1-t)$ & Conductor exponent  \\
\hline
\hline
$(p,p,r)$ & even & $<0, \equiv0\pmod r$ & $<0, \equiv0\pmod r$ & $0$\\
\cline{3-5}
& & $<0, \not\equiv 0\pmod r$ & $<0, \not\equiv 0\pmod r$ & 2\\
\cline{3-5}
& & $>0$ & 0 & $1$\\
\cline{3-5}
& & 0 & $>0$ & $1$\\
\cline{2-5}
 & odd & $\le-4$, $\equiv -2 \pmod r$ & $\le-4$, $\equiv -2 \pmod r$ & $0$\\
 \cline{3-5}
 & & $\le-4$, $\not\equiv -2 \pmod r$ & $\le-4$, $\not\equiv -2 \pmod r$ & $2$\\
 \hline
$(r,r,p)$ &  odd & $\ge 4, \equiv 4\pmod r$ & 0 & 0\\
\cline{3-5}
 &   & $\ge 4, \not\equiv 4\pmod r$ & 0 & 2\\
\cline{3-5}
& & 0 & $\ge 4, \equiv 4\pmod r$ & 0\\
\cline{3-5}
& & 0 & $\ge 4, \not\equiv 4\pmod r$ & 2\\
\hline
$(2,r,p)$ & odd & 0 & $\ge 6, \equiv 6\pmod r$ & 0\\
\cline{3-5}
&  & 0 & $\ge 6, \not\equiv 6\pmod r$ & 2\\
\hline
$(3,5,p)$ & even & $>0, \equiv 0 \pmod 3$ & 0 & 0\\
\cline{3-5}
& & $>0, \not\equiv 0 \pmod 3$ & 0 & 2\\
\cline{3-5}
& & 0 & $>0, \equiv 0\pmod 5$ & 0\\
\cline{3-5}
& & 0 & $>0, \not\equiv 0\pmod 5$ & 2\\
\cline{3-5}
& & $<0$ & <0 & 1\\
\hline
\end{tabular}
\caption{ Summary of conductors at $\Fq_2$ up to quadratic twist}
 \label{table}
\end{table}

The paper is organized as follows. Section \ref{sec:preliminaries} provides a brief introduction to hyperelliptic curves, Jacobians, and Frey representations. In Section \ref{sec:Even-reps} we study Frey representations arising from hyperelliptic equations whose defining polynomial has even degree, while Section \ref{sec:Odd-reps} is devoted to the odd degree case.

\section{Preliminaries}
\label{sec:preliminaries}

\subsection{Hyperelliptic equations}
In this section, we summarize the basic theory of hyperelliptic equations. For a detailed study of the subject, the reader may consult \cite{Liu1,Liu2,Lockhart}.

A hyperelliptic equation $E$ over a field $K$ is an equation of the form
\begin{equation*}
  y^2 + Q(x) y = P(x),
\end{equation*}
where $Q, P \in K[x]$, $\deg Q \le g+1$, and $\deg P \le n = 2g + 2$ with
\begin{equation*}
    2 g + 1 \le \max(2 \deg Q, \deg P) \le 2 g + 2.
\end{equation*}
Let
\begin{equation*}
  R = 4 P + Q^2,
\end{equation*}
and suppose $\kappa$ is the leading coefficient of $R$. The discriminant of $E$ is given by
\begin{equation*}
   \Delta_E := \begin{cases}
    2^{-4(g+1)} \Delta(R) & \text{ if } \deg R = 2 g + 2, \\
    2^{-4(g+1)} \kappa^2 \Delta(R) & \text{ if } \deg R = 2 g + 1,
\end{cases}
\end{equation*}
where $\Delta(H)$ denotes the discriminant of a polynomial $H \in K[x]$. In particular, if $P$ is monic, $\deg P = 2 g + 1$, and $\deg Q \le g$, then
\begin{equation*}
  \Delta_E = 2^{4g} \Delta(P + Q^2/4),
\end{equation*}
using the fact that $\Delta(H)$ is homogeneous of degree $2n - 2$ in the coefficients of $H$.

A hyperelliptic equation $E$ over a field $K$ gives rise to a \textit{hyperelliptic curve} $C$ over $K$ by gluing the following two affine schemes over $K$:
\begin{align}
\label{glue-model-1}   & \Spec K[x,y]/(y^2 + Q(x) y - P(x)), \\
\label{glue-model-2}   & \Spec K[u,z]/(z^2 + T(u) z - S(u)),
\end{align}
along the open sets $x \not= 0$ and $u \not= 0$, respectively, by the relations
\begin{equation*}
  x = 1/u, \quad y = z/u^{g+1},
\end{equation*}
where
\begin{equation*}
\label{relation-model}
    P(x) = u^{2g+2} S(u), \quad Q(x) = u^{g+1} T(u).
\end{equation*}
The curve $C$ is non-singular if and only if $\Delta_E \not= 0$ and we refer to the affine schemes in \eqref{glue-model-1}--\eqref{glue-model-2} as affine patches of $C$.

If $K$ has characteristic $0$, we say that $C$ has \textit{odd} (resp.\ \textit{even}) degree accordingly as the degree of $R$ as above is odd (resp.\ even).

\begin{lemma}\label{lemma:changeofvariables}
Let $E: \; y^{2} + Q(x)y = P(x)$ and 
$\widetilde{E}: \; Y^{2} + \widetilde{Q}(X)Y = \widetilde{P}(X)$ 
be two hyperelliptic equations describing the curve $C/K$. 
The coordinates $(x,y)$ and $(X,Y)$ are related by a change of variables
\[
x = \frac{aX+b}{cX+d}, 
\qquad \text{and} \qquad 
y = \frac{eY+R(X)}{(cX+d)^{g+1}},
\]
with $a,b,c,d,e \in K$, $R(X)\in K[X]$, and $ad-bc,e \neq 0$. 
We have the equality between discriminants
\[
\Delta_{\widetilde{E}} 
= e^{-4(2g+1)}(ad-bc)^{2(g+1)(2g+1)} \, \Delta_E.
\]
\end{lemma}

\subsection{Jacobian criterion in characteristic $2$}

Let $C$ be the hyperelliptic curve of genus $g$ given by
\begin{equation}\label{eq:C}
    y^2 + y Q(x) = P(x),
\end{equation}
over a field $k$ of characteristic $2$.


The following are useful criteria for the determining the semistability of $C$ \cite{wewers}.

\begin{lemma}\label{lemma:singularity}
Let $(a,b)$ be in an affine patch of $C$. Then $(a,b)$ is a singular point if and only if
\begin{align}
\label{sing1} Q(a) &= 0, \\
\label{sing2} P'(a)^2 &= Q'(a)^2 P(a). 
\end{align}
\end{lemma}
\begin{proof}
By the Jacobian criterion $(a,b)$ is a singular point of $C$ if and only if
\begin{align*}
  b^2 + b Q(a) & = P(a) \\
  Q(a) & = 0 \\
  b Q'(a) & = P'(a). 
\end{align*}
Thus, if $(a,b)$ is a singular point, then \eqref{sing1}--\eqref{sing2} hold.

Conversely, if \eqref{sing1}--\eqref{sing2} hold, the result follows from the fact that $(a,b)$ lies on $C$, together with the computation of the partial derivative of (\ref{eq:C}) with respect to $x$ at $(a,b)$.
\end{proof}

\begin{lemma}
Let $(a,b)$ be in an affine patch of $C$ and suppose $(a,b)$ is a singular point. Then $(a,b)$ is an ordinary double point if and only if $Q'(a) \not= 0$.
\end{lemma}
\begin{proof}
We recall the argument in \cite{wewers}. Assume $(a,b)$ is a point in the affine part of $C$ which is singular. The tangent cone at $(a,b)$ is
\begin{equation*}
   (y+b)^2 + Q'(a)(y+b)(x+a) + c(x+a)^2,
\end{equation*}
where $c$ is the coefficient of $x^2$ in the Taylor series expansion of $P$ at $x = a$. As $k$ has characteristic $2$, the quadratic form is nondegenerate if and only if $Q'(a) \not= 0$.
\end{proof}

\begin{lemma}
\label{semistable2}
Let $(a,b)$ be in an affine patch of $C$. Then $(a,b)$ is a smooth or ordinary double point if and only if the following does not hold:
\begin{equation*}
  Q(a) = Q'(a) = P'(a) = 0.
\end{equation*}
\end{lemma}
\begin{proof}
If $(a,b)$ is smooth, then $Q'(a) \not= 0$ or $P'(a)^2 \not= Q'(a)^2 P(a)$. In the latter case, we cannot have $Q'(a) = 0 = P'(a)$ so at least one of $Q'(a) = 0$ or $P'(a) = 0$. If $(a,b)$ is ordinary double, then $Q'(a) \not= 0$.

If $(a,b)$ is not smooth and not ordinary double, then $Q(a) = Q'(a) = 0$ and this forces $P'(a) = 0$.
\end{proof}

A smooth or ordinary double point is called a \textit{semistable} point.

\begin{corollary}
Let $(a,b)$ be in an affine patch of $C$ and assume $Q(a) = Q'(a) = 0$. If $P'(a) = 0$ then $(a,b)$ is a singular point.
In particular, if $a$ is a double root of $P(x)$ and $C$ is given by $y^2 = P(x)$, then $(a,b)$ is a singular point which is not semistable.
\end{corollary}

\subsection{Frey representations and $\GL_2$-type abelian varieties}
\label{sec:Freyreps}

The following definition is a slight modification of Darmon’s original one (see \cite[Definition 1.1]{Darmonprogram}) adapted in \cite[Definition 2.1]{BCDFvol1}. Here $K$ is a number field and $\bar{\F}_p$ is a fixed algebraic closure of the finite field with $p$ elements. Recall that the absolute Galois group of $\bar{K}(t)$,  denoted by $G_{\bar{K}(t)}$, sits  inside $G_{K(t)}$, the absolute Galois group of $K(t)$, as a normal subgroup.

\begin{definition}
     A Frey representation in characteristic \( p \) over \( K(t) \) of signature \( (n_i)_{i=1, \dots, m} \) with respect to the points \( (t_i)_{i=1, \dots, m} \), where \( n_i \in \mathbb{N} \) and \( t_i \in \mathbb{P}^1(K) \), is a Galois representation
\[
\rho : G_{K(t)} \to \GL_2(\bar{\mathbb{F}}_p)
\]
satisfying:
\begin{enumerate}
    \item the restriction \( \rho|_{G_{\bar{K}(t)}} \) has trivial determinant and is irreducible,
    \item the projectivization \( \mathbb{P}{\rho}|_{G_{\bar{K}(t)}} : G_{\bar{K}(t)} \to \PSL_2(\bar{\F}_p) \) of $\rho|_{G_{\bar{K}(t)}}$ is unramified outside the \( t_i \),
    \item the projectivization \( \mathbb{P}{\rho}|_{G_{\bar{K}(t)}}\) maps the inertia subgroups at \( t_i \) to subgroups of \( \PSL_2(\mathbb{F}) \) generated by elements of order \( n_i \), respectively, for \( i = 1, \dots, m \).
\end{enumerate}
\end{definition}

As a natural source of Frey representations, we will work with abelian varieties of $\GL_2$-type.
\begin{definition}
    Let \( A \) be an abelian variety over a field \( L \) of characteristic 0. We say that \( A/L \) is of \( \GL_2 \)-type (or \( \GL_2(K) \)-type) if there is an embedding \( K \hookrightarrow \text{End}_L(A) \otimes_{\mathbb{Z}} \mathbb{Q} \) where \( K \) is a number field with \( [K : \mathbb{Q}] = \dim A \).
\end{definition}

An abelian variety of $\GL_2$-type is so called precisely because it induces a strictly compatible system of two-dimensional Galois representations \cite[Section 11.10]{Shimura}. In the present article, we work with Jacobians $\mathcal{J}(t)$ of different families of rational hyperelliptic curves $\mathcal{C}(t)$ that become of $\GL_2$-type over a totally real number field $K$. Therefore, we will consider $\mathcal{J}(t)$ always as a variety over $K$. 

If $\id{p}$ is a prime in $K$ above $p$, then the residual representation of the $\id{p}$-member of the family of Galois representations  associated to $\mathcal{J}(t)$ is the Frey representation that we want to consider. This will be denoted by $\bar{\rho}_{\mathcal{J}(t),\id{p}}:G_{K(t)}\to\GL_2(\bar{\F}_p)$.

Let $r\ge3$ be a prime integer. Throughout this paper $\zeta_r$ will denote a fixed primitive $r$-th root of unity. For all $1\le j \le r-1$, we let $\omega_j=\zeta_r^j+\zeta_r^{-j}$ and $K=\Q(\omega_1)=\Q(\zeta_r)^+$, the maximal totally real subextension of the cyclotomic field $\Q(\zeta_r)$. Let 
\begin{equation*}
    \label{eqn:h}
h(x)=\prod_{j=1}^{\frac{r-1}{2}}(x-\omega_j)
\end{equation*}
be the minimal polynomial of $\omega_1$ and, following Darmon's  \cite{Darmonprogram}, let \[f(x)=(-1)^\frac{r-1}{2}xh(2-x^2).\]

Let $J(s)$ be the Jacobian of 
\[C(s) : y^2=f(x)+s.\]

\begin{proposition}
Let $\mathfrak{p}$ be a prime ideal in $K$ above a prime $p$. Then the representation 
$\overline{\rho}_{J(s),\mathfrak{p}}$ is a Frey representation of signature $(p,p,r)$ with respect 
to the points $(-2,2,\infty)$ 
\end{proposition}
\begin{proof}
    See \cite[p.\ 420]{Darmonprogram}.
\end{proof}

\subsubsection{Signature $(p,p,r)$} If we take $t$ another variable satisfying $s=2-4t$ we get $J^-_r(t)$, which gives rise to a Frey representation of signature $(p,p,r)$ with respect to the points $(0,1,\infty)$.

\subsubsection{Signature $(r,r,p)$}     In \cite[Section 2.4]{BCDFvol1} it has been  proved that if we take $t$ to be another variable satisfying
\begin{equation*} 
\frac{1}{s^2-4} = t(t-1)
\end{equation*}
and we take $\alpha$ to be the square-root of $t(t-1)$ then, the twist by $\alpha$ of the base change of $C(s)$ to $K(s,t)$ has a model over $K(t)$, namely the curve we denote by $C_{r,r}^-(t)$. Its Jacobian $J_{r,r}^-(t)$ gives rise to a Frey representation of signature $(r,r,p)$ with respect to the points $(0,1,\infty)$ \cite[Theorem 2.8]{BCDFvol1}.


\subsubsection{Signature $(2,r,p)$}

Let $J_{2,r}^-(t)$ be the Jacobian of the hyperelliptic curve
\begin{align*}
H_{2,r}^-(t) : y^2 & = (-t(t-1))^\frac{r-1}{2}xh\Big(2-\frac{x^2}{t(t-1)}\Big) + 2(t-1)^\frac{r-1}{2}t^\frac{r+1}{2},
\end{align*}
which has discriminant
\begin{equation}\label{eq:disc-2rp}
   \Delta(H_{2,r}^-(t))      = (-1)^{\frac{r-1}{2}} 2^{3(r-1)} r^r t^\frac{r(r-1)}{2}(t-1)^\frac{(r-1)^2}{2}. 
\end{equation}

The argument to prove that $J_{2,r}^-(t)$ is of $\GL_2(K)$-type follows the above strategy. 
 Let $K(s,t)$ be the function field where $s$ and $t$ are related by 
\begin{equation*} 
\frac{1}{s^2-4} = \frac{t-1}{4}.
\end{equation*}

We can see both $K(s)$ and $K(t)$ as subfields of $K(s,t)$. Define $\alpha$ as the square-root of $t(t-1)$. Then,  
by twisting by $\alpha$  the base change of $C(s)$ to $K(s,t)$ we get the model $H_{2,r}^-(t)$.

Replicating the arguments in \cite[Theorem 2.8]{BCDFvol1}, we have that there is an embedding $K\hookrightarrow\End_{K(t)}(J_{2,r}^-(t))\otimes\Q$ and that for every specialization of $t$ to $t_0\in\PP^1(K)-\{0,1,\infty\}$ it is well-defined.

The ramification set $\{\pm 2,\infty\}$ of $C(s)$ corresponds to the ramification set $\{\infty,1\}$ for $H_{2,r}^-(t)$. In addition, from $(\ref{eq:disc-2rp})$ we see that $0$ is a point of ramification for $H_{2,r}^-(t)$. Since $K(s,t)/K(s)$ is a degree two extension,  and $t=0$ corresponds to $s=0$ (only one point), the projective order of inertia at $t=0$ is $2$. The orders at $t = 1$ and $t = \infty$ are $r$ and $p$, respectively, from the correspondence of points. Hence, $J_{2,r}^-(t)$ gives rise to a Frey representation  of signature $(2,r,p)$ with respect to the points $(0,1,\infty)$.

\subsection{Hyperelliptic curves with potentially good reduction at 2}

Let $K$ be a number field and $\id{q}_2$ be a prime in $K$ above $2$. Let $C/K$ be a hyperelliptic curve and $J$ be its Jacobian. 

Over this section, assume that $C/K_{\id{q}_2}$ acquires good reduction over a finite extension with ramification index a prime number $r\ge3$ and that there is no unramified extension of $K_{\id{q}_2}$ where $C$ or $J$ attains good reduction.

Assume that $J/K$ is of $\GL_2(K)$-type and denote by $\{\rho_{J,\lambda}\}$ its associated strictly compatible system of two-dimensional Galois representations of $G_K$, the absolute Galois group of $K$. Let $I_{\id{q}_2}$ be the inertia subgroup of $G_K$ at $\id{q}_2$.

\begin{proposition}
\label{prop:inertialtype}
The inertial type of $J/K$ at $\mathfrak{q}_2$ is a principal series if $r \mid \#\F_{\mathfrak{q}_2}^\times$ and supercuspidal otherwise. Moreover, for all $\lambda \nmid 2$ in $K$, the image of inertia $\rho_{J,\lambda}(I_{\mathfrak{q}_2})$ is cyclic of order $r$.
\end{proposition} 

\begin{proof} This follows exactly as in \cite[Proposition 5.3]{BCDFvol1}, but we include the proof here for convenience.

We know that $J/K_{\mathfrak{q}_2}$ obtains good reduction over any extension $L/K_{\mathfrak{q}_2}$ with ramification degree $r$. In particular, this shows that $\rho_{J,\lambda}|_{D_{\mathfrak{q}_2}}$ becomes unramified over $L$, hence the inertial type of $J/K$ is not Steinberg at $\mathfrak{q}_2$. Moreover, the inertial type of $J/K$ at $\mathfrak{q}_2$ is a principal series if and only if there is an abelian extension $L/K_{\mathfrak{q}_2}$ with ramification degree $r$ and supercuspidal otherwise. By local class field theory, such an extension exists if and only if $r$ divides $\#\F_{\mathfrak{q}_2}^\times$. This proves the first statement.

Also, from Proposition \ref{prop:goodreductionextension}, there is no unramified extension of $K_{\mathfrak{q}_2}$ over which $J$ has good reduction, hence $\#\rho_{J,\lambda}(I_{\mathfrak{q}_2}) \ne 1$ and divides $r$. The conclusion follows.
\end{proof}

Let $\id{p}$ be an odd prime in $K$.  Consider $\bar{\rho}_{J,\id{p}}$, the residual representation of the $\id{p}$-member of the strictly compatible system.

\begin{proposition} \label{prop:cond-at-2-odd}
     The conductor exponent of $\bar{\rho}_{J,\id{p}}$ at $\id{q}_2$ equals $2$.
\end{proposition}
\begin{proof}
From Corollary \ref{prop:inertialtype} we have that the inertial type at $\id{q}_2$ of $\rho_{J,\id{p}}$ is either principal series or supercuspidal, and that $\rho_{J,\id{p}}(I_{\id{q}_2})$ is cyclic of order $r$. Then, the proof follows as in \cite[Theorem 5.16]{BCDFvol1}.
\end{proof}

\section{Even degree cases}
\label{sec:Even-reps}
In this section, we study Frey representations of signature $(p,p,r)$ and $(3,5,p)$. The abelian varieties giving rise to these representations are Jacobians $\mathcal{J}(t)$ of a family of rational hyperelliptic curves parametrized by $t\in\Q$, and are of $\GL(K)$-type, where $K=\Q(\zeta_r)^+$ and $K=\Q(\zeta_5)^+$, respectively. 
Let $\id{p}$ and $\id{q}_2$ be primes in $K$ above $p$ and $2$, respectively. In this section we compute the conductor of the Frey representation $\bar{\rho}_{\mathcal{J}(t),\id{p}}$ at $\id{q}_2$.

Recall for the prime $2$, three cases arise: $v_2(t)>0$, $v_2(1-t)>0$, or $v_2(t)<0$. Moreover, $v_2(t)<0$ if and only if $v_2(1-t)<0$, and in this situation we have $v_2(1-t)=v_2(t)$. Therefore, we will naturally split the analysis into these three cases.

\subsection{The Jacobian $J_r^+(t)$}
Keeping the notation of Section \ref{sec:Freyreps}, let $J_r^+(t)$ be the Jacobian of the hyperelliptic curve
\begin{align}
 C_r^+(t) : y^2 & = (x+2)(f(x) + 2 - 4t),
\end{align}
introduced in \cite{Darmonprogram}. By Theorem \cite[Theorem 1.10]{Darmonprogram} it gives rise to a Frey representation of signature $(p,p,r)$.
The discriminant of $C_r^+(t)$ equals $\Delta(C_r^+(t))=2^{2(r+1)}r^rt^\frac{r+3}{2}(1-t)^\frac{r-1}{2}$.

In \cite{azon2025}, the author studies the conductor exponent at $\id{q}_2$ using the sophisticated new theory developed in \cite{Gehrunger2025}. In particular, in the case where $J_r^+(t)$ has potentially toric reduction at $\id{q}_2$, certain restrictions are imposed on the possible values of $t$ (see Hypothesis~3 loc.\ cit.). In contrast, using only elementary arguments, we will provide a complete description for all possible parameter values.  

Using the identity $f(x)+2 = (x+2) h(-x)^2$ (see \cite[p.\ 11]{Darmonprogram}) we get 
\begin{align*}
  y^2 
 = (x+2)^2 h(-x)^2 - 4t (x+2). 
\end{align*}
Making the substitution  $y \rightarrow 2 y + (x+2)h(-x)$ we obtain the model
\begin{equation} \label{eq:model-ppr}
  y^2 + (x+2) h(-x) y =  - t(x+2),
\end{equation}
whose discriminant, by Lemma \ref{lemma:changeofvariables}, equals 
\begin{equation}\label{eq:Disc-prr}
    \Delta=r^rt^\frac{r+3}{2}(1-t)^\frac{r-1}{2}.
\end{equation}

\subsubsection{Case $v_2(t) < 0$}  
 Let us prove that the curve acquires good reduction over $K_{\id{q}_2}((1/t)^\frac{1}{r})$. To simplify notation, let $u=(1/t)^\frac{1}{r}$, and so $t = 1/u^r$. 
Making the substitution $x\to x/u$, $y\to y/u^\frac{r+1}{2}$ in (\ref{eq:model-ppr}) we get the integral model
 \begin{equation}
 y^2 + (x+2u) u^\frac{r-1}{2}h(-x/u) y = -(x+2u).
\end{equation}
From (\ref{eq:Disc-prr}) and Lemma \ref{lemma:changeofvariables}, its discriminant is a unit, so the above is a model with good reduction.

\subsubsection{Case $v_2(t) > 0$}
From Lemma \ref{lemma:singularity}, the model (\ref{eq:model-ppr}) has the singular points $(-2,0)$ and $(-\omega_j,0)$ for all $1\le j \le (r-1)/2$, and by Lemma \ref{semistable2} it   has semistable reduction at $\id{q}_2$.  

\subsubsection{Case $v_2(1-t) > 0$}
Similarly as in the previous analysis, we can use the identity $f(x)-2 = (x-2) h(-x)^2$ and get the model
\begin{align*}
  y^2  = (x^2-4) h(-x)^2 + 4(1-t) (x+2). 
\end{align*}
Making the substitution $ y \rightarrow 2 y + x h(-x)$
we obtain the model
\begin{equation}\label{eq:newmodel}
  y^2 + x h(-x) y = - h(-x)^2 + (1-t) (x+2).
\end{equation}
Once again, by Lemma \ref{lemma:singularity}, we have that $(-\omega_j,0)$ are singular points of the model (\ref{eq:newmodel}), which has semistable reduction at $\id{q}_2$, by Lemma \ref{semistable2}.

\begin{proposition}
    The conductor exponent of 
    $\bar{\rho}_{J_r^+(t),\id{p}}$ at $\id{q}_2$ equals 
    \[
    \begin{cases}
    0 & \text{if } v_2(t)<0 \text{ and } v_2(t)\equiv 0\pmod r,\\
    2 & \text{if } v_2(t)<0 \text{ and } v_2(t)\not\equiv 0\pmod r,\\
    1 & \text{if } v_2(t(1-t))>0.
    \end{cases}
    \]
\end{proposition}
\begin{proof}
    The first and the third case follow directly from the previous analysis. In the second case, the curve acquires good reduction over  an extension with ramification index $r$ and then the result follows from Proposition \ref{prop:cond-at-2-odd}.
\end{proof}

\subsection{The Jacobian $J_{3,5}^+(t)$} Let $J_{3,5}^+(t)$ be the Jacobian of the hyperelliptic curve
\begin{equation*}\label{eq:curve35p}
  H_{3,5}^+(t) : y^2 + y(x^3+t(1-t)^2) = 2t(1-t)^2x^3 + 3t^2(1-t)^3x + t^2(1-t)^4,
\end{equation*}
whose discriminant equals $\Delta( H_{3,5}^+(t)) = 3^6 5^5 t^{10}(t-1)^{18}.$

Setting $t = Aa^3/Cc^p$ and $1-t = Bb^5/Cc^p$ 
gives rise to an isomorphism to a Frey hyperelliptic curve for signature $(3,5,p)$, namely
\begin{equation*}
    y^2 + y(x^3 + AB^2b) = 2AB^2bx^3 + 3A^2B^3ax + A^2B^4b^2,
\end{equation*}
having discriminant $\Delta=3^65^5A^{10}B^{18}(Aa^3+Bb^5)^2$.
As far as we know, this is the first time that a Frey hyperelliptic curve for signature $(3,5, p)$ has been considered. The forthcoming work \cite{PacettiVillagra35p} 
explains how to construct the curve $H_{3,5}^+(t)$ from the theory of hypergeometric motives, a framework that was recently incorporated into the modular method to address generalized Fermat equations \cite{GolfieriPacetti}. 

\subsubsection{Case $v_2(t)>0$}
Let us prove that the curve acquires good reduction over $K_{\id{q}_2}(t^\frac{1}{3})$. To simplify notation, set $u=t^\frac{1}{3}$. Then, the initial model is
\begin{equation*}
   y^2 + y(x^3+u^3(1-u^3)^2) =  2u^3(1-u^3)^2 x^3 + 3 u^6 (1-u^3)^3 x +   u^6(1-u^3)^4.
\end{equation*}

Applying the change of variables $x\to ux$ and $y\to u^3y$ we get the model
\begin{equation*}
   y^2 + y(x^3+(1-u^3)^2) =  2(1-u^3)^2 x^3 + 3 u (1-u^3)^3 x +   (1-u^3)^4,
\end{equation*}
whose discriminant is a unit, by Lemma \ref{lemma:changeofvariables}.

\subsubsection{Case $v_2(1-t)> 0$}
Let us prove that the curve acquires good reduction over $K_{\id{q}_2}((1-t)^\frac{1}{5})$. To simplify notation, set $u=(1-t)^\frac{1}{5}$. Then, the initial model is
\begin{equation*}
   y^2 + y(x^3+(1-u^5)u^{10}) =  2u(1-u^5)u^{10} x^3 + 3 (1-u^5)^2u^{15} x +   u^2(1-u^5)^2u^{20}.
\end{equation*}

Applying the change of variables $x\to u^3x$ and $y\to u^9y$ we get the model
\begin{equation*}
   y^2 + y(x^3+(1-u^5)u) =  2(1-u^5)u x^3 + 3 (1-u^5)^2 x +   (1-u^5)^2u^2, 
\end{equation*}
whose discriminant is a unit, by Lemma \ref{lemma:changeofvariables}.

\subsubsection{Case $v_2(t) < 0$} Recall that in this case we have $v_2(1-t)=v_2(t)$. Under this assumption,  we can see that the curve has toric reduction over $K_{\id{q}_2}(1/t)$. 

Set $u=(1-t)/t$, which is a unit. By applying the change of variables $x\to tx$ and $y\to t^3y$ we get the model
\begin{equation*}
   y^2 + y(x^3+u^2) =  2u^2 x^3 + 3 u^3 x +  u^4
\end{equation*}
which a semistable model by Lemma \ref{semistable2}, whose special fiber is given by
\begin{equation*}
    y^2 + y (x^3-1) = x + 1,
\end{equation*}
with singular points at $x = \zeta_3$, by Lemma \ref{lemma:singularity}.

For the purposes of this work, we assume that $J_{3,5}^+(t)$ is of $\GL_2(K)$-type for $K = \Q(\zeta_5)^+$ (this is also proved in \cite{PacettiVillagra35p}). Under this assumption, we can follow the strategy in the proof of \cite[Theorem 1.10]{Darmonprogram} to use the above arguments to show that the attached Galois representation $\rhobar_{J_{3,5}^+(t),\Fp} : G_{K(t)} \rightarrow \GL_2(\bar{\F}_p)$  is a Frey representation of signature $(3,5,p)$ with respect to the points $(0,1,\infty)$.

\begin{proposition} Assume that $J_{3,5}^+(t)$ is of $\GL_2(\Q(\zeta_5^+))$-type.
    The conductor exponent of 
    $\bar{\rho}_{J_{3,5}^+(t),\id{p}}$
    at $\id{q}_2$ equals 
    \[
    \begin{cases}
    0 & \text{if } v_2(t(1-t))>0, \ v_2(t)\equiv 0\pmod 3 \text{ and } v_2(1-t)\equiv 0\pmod 5,\\
    2 & \text{if } v_2(t(1-t))>0 \text{ and } v_2(t)\not\equiv 0\pmod 3 \text{ or } v_2(1-t)\not\equiv 0\pmod 5,\\
    1 & \text{if } v_2(t)<0.
    \end{cases}
    \]
\end{proposition}
\begin{proof}
    The first and the third case follow directly from the previous analysis. In the second case, the curve acquires good reduction over  an extension with ramification index $3$ if $v_2(t)\not\equiv0\pmod3$ and $5$ if $v_2(1-t)\not\equiv0\pmod5$. Then the result follows from Proposition \ref{prop:cond-at-2-odd}.
\end{proof}

\section{Odd degree cases}
\label{sec:Odd-reps}

In this section we study  Frey representations of signature $(p,p,r)$, $(r,r,p)$ and $(2,r,p)$, where $r\ge3$ is a fixed prime and $p\neq r$ is a prime variable. The abelian varieties giving rise to these representations are Jacobians $\mathcal{J}(t)$ of a family of rational hyperelliptic curves parametrized by $t\in\Q$, and are of $\GL_2(K)$-type, where $K=\Q(\zeta_r)^+$. Let $\id{p}$ and $\id{q}_2$ be a primes in $K$ above $p$ and $2$, respectively. We focus on computing the conductor exponent  of the Frey representation $\bar{\rho}_{\mathcal{J}(t),\id{p}}$ at $\id{q}_2$
for the cases where $\mathcal{J}(t)$ has potentially good reduction at $\id{q}_2$. The case where $\mathcal{J}(t)$ has (potentially) multiplicative reduction at $\id{q}_2$ remains outside the scope of this paper for these Frey representations.


We will carry out this aim by following the framework introduced in \cite{cazorla2025}.


\subsection{A common framework}
Keeping the notation of Section \ref{sec:Freyreps}, consider the hyperelliptic curve
\[C=C(z,s) : y^2 = (-z)^\frac{r-1}{2}xh\left(2-\frac{x^2}{z}\right)+s,\]
where $z, s\in \Q$. $C$ is a rational curve, but we will consider its base change over $K$. Let $J=J(z,s)/K$ be its Jacobian. The discriminant of $C$ (see \cite[Lemma 3.1]{cazorla2025}) equals
 \begin{equation}\label{eq:discgeneral}
      \Delta_C = (-1)^{\frac{r-1}{2}} 2^{2(r-1)} r^r (s^2-4z^{r})^{\frac{r-1}{2}}.
 \end{equation}



\begin{lemma}\label{lemma:twist2} 
$C(\delta^2 z,\delta^r s)$ is isomorphic to the quadratic twist by $\delta$ of $C(z,s)$.

\end{lemma}


Over the  completion $K_{\id{q}_2}$ we will set the notation $\pi_K$ to be a uniformizer, $\calO_K$ the ring of integers and $v_K$ the corresponding valuation.

\begin{proposition}\label{prop:goodreductionextension} Suppose $v_K(z^r) \ge v_K(s^2)+4$. Let $L/K_{\id{q}_2}$ be a finite extension with ramification index $r$. Then, up to a quadratic twist by $\pm \pi_K$, $C$ has good reduction over $L$.
\end{proposition}

\begin{proof}
Let $\pi_L$ be the uniformizer of the extension, $\calO_L$ its ring of integers and $v_L$ the corresponding valuation. It follows from  Lemma \ref{lemma:twist2} that, up to a quadratic twist by $\pm \pi_K$, we may assume $v_K(s)$ is even and that $s/\pi_K^{v_K(s)} \equiv 1\pmod{\pi_K^2}$.

Applying the substitution $x\to\pi_L^{v_K(s)}x$, $y\to \pi_L^{r\frac{v_K(s)}{2}}y $ and dividing out by $\pi_L^{rv_K(s)}$ we get the model
\[ y^2 = x^r + \sum_{k=1}^{\frac{r-1}{2}}c_k\left(\frac{z}{\pi_L^{v_K(s^2)}}\right)^kx^{r-2k}+\frac{s}{\pi_L^{rv_K(s)}}.\]
Now, making the substitution $x\to \pi_Lx$, $y\to\pi_L^ry+1$ and dividing out by $\pi_L^{2r}$ we get the model
\begin{align}\label{eq:mododeloverK}
    y^2 + \frac{2}{\pi_L^r}y = x^r + \sum_{k=1}^{\frac{r-1}{2}}c_k\left(\frac{z}{\pi_L^{v_K(s^2)+4}}\right)^kx^{r-2k}+\frac{\frac{s}{\pi_L^{rv_K(s)}}-1}{\pi_L^{2r}},
\end{align}
where $v_L(z)=v_K(z^r)\ge  v_K(s^2) + 4$ by hypothesis and $v_L(2/\pi_L^{rv_K(s)}-1)=rv_K(2/\pi_K^{v_K(s)}-1)\ge 2r$. Therefore (\ref{eq:mododeloverK}) is defined over $\calO_L$. Moreover, 
from (\ref{eq:discgeneral}) and Lemma \ref{lemma:changeofvariables}, its discriminant equals
\[ (-1)^{\frac{r-1}{2}} \frac{2^{2(r-1)}}{\pi_L^{2r(r-1)}} \cdot r^r \cdot \frac{(s^2-4z^{r})^{\frac{r-1}{2}}}{\pi_L^{r(r-1)v_K(s)}},\]
which is a unit in $\calO_L$.
\end{proof}

\begin{remark}\label{remark:goodred}
    Notice that, if $v_K(s^2)+4\equiv0\pmod r$, the model (\ref{eq:mododeloverK}) is defined over $K_{\id{q}_2}$. Otherwise, there is no unramified extension where $C$ or $J$ attains good reduction.
\end{remark}

\subsection{The Jacobian $J_{r}^-(t)$}

Let $J_{r}^-(t)$ be the Jacobian of the hyperelliptic curve
\begin{align*}
C_{r}^-(t) : y^2 & = f(x) + 2 - 4t \\
& = (-1)^\frac{r-1}{2}xh(2-x^2) + 2 - 4t,
\end{align*}
introduced in \cite{Darmonprogram}. As mentioned in Section \ref{sec:Freyreps}, it gives rise to a Frey representation of signature $(p,p,r)$.

\begin{corollary}
    Suppose $v_2(t)\le - 4 $. Then, up to a quadratic twist of the curve, the conductor exponent of $\bar{\rho}_{J^-_r(t),\id{p}}$ at $\id{q}_2$ equals
   \[ \begin{cases}
       0 & \text{if } v(t)\equiv -2 \pmod r,\\
        2 & \text{if } v(t)\not\equiv -2 \pmod r.
    \end{cases}
    \]
\end{corollary}
\begin{proof}
    In this case we have $z=1$ and $s=2-4t$. Since
    \[v_2(z^r)= 0 = -4 + 4 \ge v(s^2) + 4,\]
    the result follows from Proposition \ref{prop:cond-at-2-odd} and Remark \ref{remark:goodred}.
\end{proof}

\subsection{The Jacobian $J_{r,r}^-(t)$}

Let $J_{r,r}^-(t)$ be the Jacobian of the hyperelliptic curve
\begin{align*}
H_{r,r}^-(t) : y^2 & = (-t(t-1))^\frac{r-1}{2}xh\Big(2-\frac{x^2}{t(t-1)}\Big) + ((t-1)t)^\frac{r-1}{2}(2t-1),
\end{align*}
considered in \cite{BCDFvol1}. As mentioned in Section \ref{sec:Freyreps}, it gives rise to a Frey representation of signature $(r,r,p)$. The curve $H_{r,r}^-(t)$ has discriminant
\begin{equation*}
    \Delta_{H_{r,r}^-(t)}=(-1)^{\frac{r-1}{2}} 2^{2(r-1)} r^r (t(t-1))^\frac{(r-1)^2}{2}.
\end{equation*}

\begin{corollary}
    Suppose $v_2(t(t-1))\ge 4$. Then, up to a quadratic twist of the curve,  the conductor exponent of $\bar{\rho}_{J^-_{r,r}(t),\id{p}}$ at $\id{q}_2$ equals
   \[ \begin{cases}
       0 & \text{if } v(t(t-1))\equiv 4 \pmod r,\\
        2 & \text{if } v(t(t-1))\not\equiv 4 \pmod r.
    \end{cases}
    \]
\end{corollary}
\begin{proof}
    In this case we have $z=t(t-1)$ and $s=((t-1)t)^\frac{r-1}{2}(2t-1)$. Since
    \[v_2(z^r)=rv(t(t-1)) = (r-1)v(t(t-1)) + v(t(t-1))\ge (r-1)v(t-1) + 4 = v(s^2) + 4,\]
    the result follows from Proposition \ref{prop:cond-at-2-odd} and Remark \ref{remark:goodred}.
\end{proof}


\subsection{The Jacobian $J_{2,r}^-(t)$}

Let $J_{2,r}^-(t)$ be the Jacobian of the hyperelliptic curve
\begin{align*}
H_{2,r}^-(t) : y^2 & = (-t(t-1))^\frac{r-1}{2}xh\Big(2-\frac{x^2}{t(t-1)}\Big) + 2(t-1)^\frac{r-1}{2}t^\frac{r+1}{2},
\end{align*}
with discriminant given in (\ref{eq:disc-2rp}). Setting $t=Aa^2/Cc^p$ we recover the Frey hyperelliptic curve associated to equation (\ref{eq:GFE}) for signature $(2,r,p)$ up to a twist (see e.g.\ \cite[Section 6.3]{cazorla2025}). Indeed, as mentioned in Section \ref{sec:Freyreps}, $J_{2,r}^-(t)$ gives rise to a Frey representation of signature $(2,r,p)$.


\begin{corollary}
    Suppose $v_2(t-1)\ge 6$. Then, up to a quadratic twist of the curve, the conductor exponent of $\bar{\rho}_{J^-_{2,r}(t),\id{p}}$ at $\id{q}_2$ equals
   \[ \begin{cases}
       0 & \text{if } v(t-1)\equiv 6 \pmod r,\\
        2 & \text{if } v(t-1)\not\equiv 6 \pmod r.
    \end{cases}
    \]
\end{corollary}
\begin{proof}
    Suppose $v_2(t-1)\ge 6$ and so $v_2(t)=0$. In this case we have $z=t(t-1)$ and $s=2(t-1)^\frac{r-1}{2}t^\frac{r+1}{2}$. Since
    \[v_2(z^r)=rv(t-1) = (r-1)v(t-1) + v(t-1)\ge (r-1)v(t-1) +6 = v(s^2) + 4,\]
    the result follows from Proposition \ref{prop:cond-at-2-odd} and Remark \ref{remark:goodred}.
\end{proof}

\bibliographystyle{alpha}
\bibliography{notes_cond_at_2}

\end{document}